\documentclass[twoside]{article}
\usepackage[english]{babel}
\usepackage{amssymb,amsmath,amsthm}
\newtheorem{de}{Definition}[section]
\newtheorem{theor}{Theorem}[section]
\newtheorem{pr}{Proposition}[section]
\newtheorem{lm}{Lemma}[section]

\begin{document}

\begin{center}
{\Large \textbf{On one dimensional Leibniz central extensions\\ of a naturally graded filiform Lie algebra}}\\

I.S. Rakhimov$^{1}$, \ \ Munther.A.Hassan$^{2}$

 Institute for Mathematical Research (INSPEM) $\&$

Department of Mathematics, FS,UPM, 43400,

Serdang, Selangor Darul Ehsan,(Malaysia)\\[0pt]

\emph{isamiddin@science.upm.edu.my $\&$ risamiddin@mail.ru}\\
\emph{munther\_abd@yahoo.com}

\end{center}

\author{ I.S. Rakhimov$^{1}$,M.A.Hassan $^{2},$ $}
\title{\mathbf{On one dimensional Leibniz central extensions of a naturally graded filiform Lie algebra}}

\begin{abstract}
This paper deals with the classification of Leibniz central
extensions of a naturally graded filiform Lie algebra. We choose a
basis with respect to that the table of multiplication has a
simple form. In low dimensional cases isomorphism classes of the
central extensions are given. In parametric family orbits cases
invariant functions (orbit functions) are provided.
\end{abstract}

\medskip



\thispagestyle{empty}

\section{Introduction}

\qquad  Leibniz algebras were introduced by J.-L.Loday
\cite{L},\cite{LP}. (For this reason, they have also been called
\textquotedblleft Loday algebras\textquotedblright). A
skew-symmetric Leibniz algebra is a Lie algebra. The Leibniz
algebras play an important role in Hochschild homology theory, as
well as in Nambu mechanics. The main motivation of J.-L.Loday to
introduce this class of algebras was the search of an
\textquotedblleft obstruction\textquotedblright \ to the
periodicity of algebraic $K-$theory. Beside this purely algebraic
motivation some relationships with classical geometry,
non-commutative geometry and physics have been recently
discovered. Leibniz algebras appear to be related in a natural way
to several topics such as differential geometry, homological
algebra, classical algebraic topology, algebraic K-theory, loop
spaces, noncommutative geometry, quantum physics etc., as a
generalization of the corresponding applications of Lie algebras
to these topics. It is a generalization of Lie algebras. K.A.
Umlauf (1891) initiated the study of the simplest non-trivial
class of Lie algebras.  In his thesis he presented the list of Lie
algebras of dimension less than ten admitting a so-called adapted
basis (now,  Lie algebras with this property are called filiform
Lie algebras).There is a description of naturally graded complex
filiform Lie algebras as follows: up to isomorphism there is only
one naturally graded filiform Lie algebra in odd dimensions and
they are two in even dimensions. With respect to the adapted basis
table of multiplications have a simple form.  Since a Lie algebra
is Leibniz it has a sense to consider a Leibniz central extensions
of the filiform Lie algebra. The resulting algebra is a filiform
Leibniz algebra and it is in of interest to classify these central
extensions. In the present paper we propose an approach based on
algebraic invariants. The results show that this approach is quite
effective in the classification problem. As a final result we give
a complete list of the mentioned class of algebras in low
dimensions. In parametric family orbits case we provide invariant
functions to discern the orbits (orbit functions). As the next
step of the study the algebraic classification may be used in
geometric study of the algebraic variety of filiform Leibniz
algebras.

The (co)homology theory, representations and related problems of
Leibniz algebras were studied by Loday, J.-L. and Pirashvili, T.
\cite{LP}, Frabetti, A. \cite{F} and others. A good survey about
these all and related problems is \cite{LFCG}.

The problems related to the group theoretical realizations of
Leibniz algebras are studied by Kinyon, M.K., Weinstein, A.
\cite{KW} and others.

Deformation theory of Leibniz algebras and related physical
applications of it, is initiated by Fialowski, A., Mandal, A.,
Mukherjee, G. \cite{FMM}.

The outline of the paper is as follows. Section 2 is a gentle
introduction to a subclass of Leibniz algebras that we are going
to investigate. Section 3 describes the behavior of parameters
under the isomorphism action (adapted changing). Sections 3.1 ---
3.5 contain the main results of the paper consisting of the
complete classification of one dimensional Leibniz central
extensions of low dimensional graded filiform Lie algebras. Here
we give complete lists of all one dimensional Leibniz central
extensions in low dimensions cases.  We distinguish the
isomorphism classes and show that they exhaust all possible cases.
For parametric family cases the corresponding invariant functions
are presented. Since the proofs from the considered cases can be
carried over to the other cases by minor changing we have chosen
to omit the proofs of them. All details of the other proofs are
available from the authors.

\section{Preliminaries}

\qquad Let $V$ be a vector space of dimension $n$ over an
algebraically closed field $K$ (char$K$=0). The bilinear maps
$V\times V\longrightarrow V\ $form a vector space $Hom(V\otimes
V,V)$ of dimension $n^{3}$, which can be considered together with
its natural structure of an affine algebraic variety over $K$ and
denoted by$\ Alg_{n}(K)\cong K^{n^{3}}$. An $n$-dimensional
algebra $L$ over $K$ may be considered as an element of$\
Alg_{n}(K)$ \ via the bilinear mapping $[\cdot,\cdot]:L\times $
$L$ $\longrightarrow L\ $defining a binary algebraic operation on
$L$. Let $\{e_{1},e_{2},...,e_{n}\}$ be a basis of $L$. Then the
table of multiplication of $L$ is represented by a point $
\{\gamma _{ij}^{k}\}\ $of the affine space $K^{n^{3}}$ as follows:

$$ [e_{i},e_{j}]=\sum_{k=1}^{n}\gamma _{ij}^{k}e_{k}$$

($\gamma _{ij}^{k}\ $are called structural constants of $L$). The
linear group $GL_{n}(K)$ acts on $\ A\lg_n(K)\ $ by

$$[x,y]_{g\ast L}=g[g^{-1}(x),g^{-1}(y)]_{L}$$
where $L\in \ A\lg_n(K),g\in GL_{n}(K)$\\

Two algebras $L_{1}$ and $L_{2}$ are isomorphic if and only if
they belong to the same orbit under this action.

\begin{de}
\bigskip\ An algebra L over a field K is said to Leibniz algebra if its
bilinear operation $[\cdot,\cdot]$ satisfies the following Leibniz
identity:
\end{de}
\begin{equation*}
\lbrack x,[y,z]]=[[x,y],z]-[[x,z],y],
\end{equation*}

Let $LB_{n}(K)$ be the subvariety of $Alg_{n}(K)$ consisting of all $n$%
-dimensional Leibniz algebras over $K$. It is invariant under the
above mentioned action of $GL_{n}(K)$. As a subset of $Alg_{n}(K)$ the set $%
LB_{n}(K)$ is specified by the system of equations with respect to
the structural constants $\gamma _{ij}^{k}$:
\begin{equation*}
\sum\limits_{\emph{l}=1}^{\emph{n}}{(\gamma _{\emph{jk}}^{\emph{l}}\gamma _{%
\emph{il}}^{\emph{m}}-\gamma _{\emph{ij}}^{\emph{l}}\gamma _{\emph{lk}}^{%
\emph{m}}+\gamma _{\emph{ik}}^{\emph{l}}\gamma
_{\emph{lj}}^{\emph{m}})}=0
\end{equation*}
Further all algebras are assumed to be over the field of complex
numbers $\mathbb{C}$.

\begin{de}Let $L$ and $V$ be Leibniz algebras. An extension $\widetilde{L}$ of $L$ by $V$ is a short exact sequence:
$$ 0 \longrightarrow V \longrightarrow \widetilde{L} \longrightarrow L \longrightarrow 0 $$ of Leibniz algebras.
\end{de}

The extension is said to be \emph{central} if the image of $V$ is
contained in the center of $\widetilde{L}$ and \emph{one
dimensional} if $V$ is.

Let $L$ be a Leibniz algebra. We put:
\begin{equation*}
L^{1}=L,\ L^{k+1}=[L^{k},L],\ k\geq 1.
\end{equation*}
\begin{de}
A Leibniz algebra L is said to be nilpotent if there exists an
integer $s\in N,$ such that
\end{de}
\begin{equation*}L^{1}\supset  L^{1}\supset ...\supset L^{s}=\{0\}.\end{equation*}
\begin{de}
A Leibniz algebra $L$ is said to be filiform if $dimL^{i}=n-i,$ where $%
n=dimL$ and $2\leq i\leq n.$
\end{de}

It is obvious that a filiform Leibniz algebra is nilpotent.

The set of all $n-$dimensional filiform Leibniz algebras we denote
as $Leib_n.$
\section{ Simplifications in $CE_{\mu_n}$}

In this section we consider a subclass of $Leib_{n+1}$ called
truncated filiform Leibniz algebras in \cite{OR}, where
motivations to study of this case also has been given. According
to \cite{OR} the table of multiplication of the truncated filiform
Leibniz algebras can be represented as follows:

$ \left\{
\begin{array}{lll}
\lbrack e_{i},e_{0}]=e_{i+1}, & 1\leq i\leq {n-1}, &  \\[1mm]
\lbrack e_{0},e_{i}]=-e_{i+1}, & 2\leq i\leq {n-1}, &  \\[1mm]
\lbrack e_{0},e_{0}]=b_{0,0}e_{n}, &  &  \\[1mm]
\lbrack e_{0},e_{1}]=-e_{2}+b_{01}e_{n}, &  &  \\[1mm]
\lbrack e_{1},e_{1}]=b _{11}e_{n}, &  &  \\[1mm]
\lbrack e_{i},e_{j}]=b_{ij}e_{n}, & 1\leq i<j\leq {n-1}, &  \\[1mm]
\lbrack e_{i},e_{j}]=-[e_{j},e_{i}], & 1\leq i<j\leq n-1, &  \\[1mm]
\lbrack e_{i},e_{n-i}]=-[e_{n-i},e_{i}]=(-1)^{i}b\,e_{n} & 1\leq
i\leq n-1. &\\ \\ b\in\{0,1\} \ \ \mbox{for odd} \ \ n \ \
\mbox{and} \ \ b=0 \ \ \mbox{for even} \ \ n,
\end{array}%
\right. $

The basis $\{e_0,e_1,...,e_{n-1},e_n\}$ leading to this
representation is said to be \emph{adapted}.

It is obvious that this is a class of all one dimensional Leibniz
central extensions of the graded filiform Lie algebra with the
composition law $[\cdot,\cdot]:$
$$\mu_n: \ \ \lbrack e_{i},e_{0}\rbrack=e_{i+1},  \ \ 1\leq i\leq {n-1},$$
with respect to the adapted basis $\{e_0,e_1,...,e_{n-1}\}.$

 \begin{de}
Let $\{e_0,e_1,...,e_n\}$ be an adapted basis of $L\in CE(\mu_n).$
Then a nonsingular linear transformation $f:L\rightarrow L$ is
said to be adapted if the basis
 $\{f(e_0),f(e_1),...,f(e_n)\} $ is adapted.
\end{de}

The set of all adapted elements of $GL_{n+1}$ is a subgroup and it
is denoted by $G_{ad}.$

Elements of $CE(\mu_n)$ represented by the above table shortly we
denote as $L=L(b_{0,0},b_{0,1},b_{1,1},...,b_{i,j})$ with $1\leq
i<j\leq {n-1}.$

Since a filiform Leibniz algebra is 2-generated the basis changing
on it can be taken as follows:

$$f(e_0)=\sum_{i=0}^{n}A_{i}e_{i} $$

$$f(e_1)=\sum_{i=0}^{n}B_{i}e_{i} $$

where $A_0(A_0B_1-A_1B_0)(A_0+A_1b)\neq0$ and let $f(L)=L'$.

The following lemma specifies the parameters
$(b_{00},b_{01},b_{11},...,b_{ij})$ of the algebra
$L=L(b_{0,0},b_{0,1},b_{1,1},...,b_{i,j}).$

\begin{lm}\emph{}Let $L \in CE(\mu_n)$.Then the following equalities hold:
\begin{enumerate}
\item
 $$b_{i+1,j}=-b_{i,j+1} \ \ \  1 \leq i,j\leq n-1 ,\ \ i+j\neq n
$$
\item
$$b_{1,2i+1}=0 \ \ \ \ \ \ \ \ \ 0 < i\leq \left[\frac{n-2}{2}\right]$$

\end{enumerate}

\begin{proof}
\emph{1}. From Leibniz Identity we will get the following identity
for
 $i,j\geq2$
\begin{eqnarray*}b_{i+1,j}\,e_n+b_{i,j+1}\,e_n&=&[e_{i+1},e_j]+[e_i,e_{j+1}]\\[1mm]
[[e_i,e_0],e_j]+[e_i,[e_j,e_0]]&=&[[e_0,e_j],e_i]-[[e_0,e_i],e_j]=[e_0,[e_j,e_i]]=0\Rightarrow
b_{i+1,j}=-b_{i,j+1}.
\end{eqnarray*}
The equality still true for $i,j\geq1$



\emph{2}.This equality is followed from the chain of equalities
\begin{eqnarray*}b_{1,2i+1}\,e_n&=&[e_1,e_{2i+1}]=[e_1,[e_{2i},e_0]]\\[1mm]
&=&[[e_1,e_{2i}],e_0]-[[e_1,e_0],e_{2i}]\\[1mm]
&=&[[e_1,e_0],e_{2i}]=[e_2,e_{2i}]=0\end{eqnarray*}


\end{proof}
\end{lm}
Consequence of this lemma we will get   $b_{i+2,i}=b_{i,i+2}=0 $
and
$$b_{i,j}=-b_{i-1,j+1}=b_{i-2,j+2}=...=(-1)^i\,b_{i-(i-1),j+(i-1)}=(-1)^i\,b_{1,j+i-1}.$$

\begin{pr}
Let $f\in G_{ad}$ if $ L\in CE(\mu_n)$ then $f$ has the following
form:

\begin{eqnarray*}e_{0}^{\prime }&=&\sum_{i=0}^{n}A_{i}e_{i} \\[1mm]
e_{i}^{\prime }&=&\sum_{k=i}^{n-1}A_{0}^{i-1}B_{k-i+1}e_{i}+(\ast
)e_{n}~\ \ \ \ \ \ \ \ \ \ \ \ \ \ \ 1\leq i\leq n-1\\[1mm]
e_{n}^{\prime }&=& A_{0}^{n-2}B_{1}(A_0+A_1b) \,e_n
\end{eqnarray*}
where $A_0B_1(A_0+A_1b)\neq0$
\begin{proof}
Note that
\begin{eqnarray}e'_i&=&f(e_{i})=[f(e_{i-1}),
f(e_0)]=\sum_{j=i}^{n-1}A_0^{i-2}(A_0B_{j-i+1}-A_{j-i+1}B_0)e_j+(*)e_n,
\ 2 \leq i \leq n-1 \\[1mm]
e'_n&=&f(e_{n})=[f(e_{n-1}),
f(e_0)]=A_0^{n-3}(A_0B_{1}-A_{1}B_0)(A_0+A_{1}b)e_n
\end{eqnarray}
$$ \mbox{Consider}\ \ [f(e_2),f(e_1)]=B_0\sum_{i=3}^{n-1}(A_0B_{i-2}-A_{i-2}B_0)e_i+(*)e_n,$$ and
equating the corresponding coefficients we get
$B_0(A_0B_{i-2}-A_{i-2}B_0)=0, \  3\leq i \leq n-1$ .Since
$A_0B_1-A_1B_0\neq0$ then this  relation is only possible if
$B_0=0.$

\end{proof}
\end{pr}

\begin{de}

The following transformations of $L$ is said to be elementary:
$$ \sigma( b, k)=\left\{\begin{array}{lll} f(e_0)=e_0, &
\\[1mm]
f(e_1)=e_1+b\,e_k, \ \ \ \ \ \ &  \quad b\in \mathbb{C},\quad 2\leq k \leq n\\[1mm]
f(e_{i+1})=[f(e_i), f(e_0)], &  \quad 1\leq i\leq n-1,  \\[1mm]
\end{array} \right.$$
$$ \tau(a,k)=\left\{\begin{array}{lll} f(e_0)=e_0+a\,e_k, \quad a\in \mathbb{C} \quad 1\leq k \leq n , &
\\[1mm]
f(e_1)=e_1, & \\[1mm]
f(e_{i+1})=[f(e_i), f(e_0)], & \quad 1\leq i\leq n-1,  \\[1mm]
\end{array} \right.$$
$$\upsilon(a,b)=\left\{\begin{array}{lll}f(e_0)=a\,e_0, &
\\[1mm]
f(e_1)=b\,e_1, &\quad a, b\in \mathbb{C^*}.
\\[1mm]
f(e_{i+1})=[f(e_i), f(e_0)], & \quad 1\leq i\leq n-1
\end{array}\right. $$
\end{de}

\begin{pr} Let $f$ be an adapted transformation of $L.$ Then it can be represented as
composition:
$$f=\tau(a_n,n)\circ\tau(a_{n-1},n-1)\circ...\circ\tau(a_2,2)\circ\sigma(b_n,n)\circ\sigma(b_{n-1},n-1)\circ...\circ\sigma(b_2,2)\circ\tau(a_1,1)
\circ\upsilon(a_0, b_1),$$

\end{pr}
\begin{proof} The proof is straightforward.
\end{proof}

\begin{pr} The transformation
$$g=\tau(a_n,n)\circ\tau(a_{n-1},n-1)\circ...\circ\tau(a_2,2)\circ\sigma(b_n,n)\circ\sigma(b_{n-1},n-1)$$

does not change the structural constants of this case
\end{pr}
So from the assertion above of proposition 3.3. we have the
adapted transformations are reduced to the transformation of the
form :

$$ \left\{\begin{array}{lll} f(e_0)=A_0\,e_0 + A_1\,e_1 &
\\[1mm]
f(e_1)=B_1\,e_1+B_2\,e_2+...+B_{n-2}\,e_{n-2}, & \\[1mm]
f(e_{i+1})=[f(e_i), f(e_0)], & \quad 1\leq i\leq n-1,  \\[1mm]
\end{array} \right.$$
where $A_0\,B_1(A_0+A_1\,b)\neq0$

Under the action of the given basis change we have

The next lemma defines the action of the adapted changing of basis
to the structural constants of algebras from $CE(\mu_n).$

\begin{lm}
Let $L\in CE(\mu_n)$ with parameters $L(\alpha)$ where
$\alpha=(b_{0,0},b_{0,1},b_{1,1},b_{1,2},b_{1,4},...,b_{1,2j})$
and $L'$ be the image of $L$ under the action of $G_{ad}.$ Then
for
parameters of $L'$ one has:\\
\begin{eqnarray*}b_{0,0}^\prime&=&\frac{A_0^2b_{0,0}+A_0A_1b_{0,1}+A_1^2b_{1,1}}{A_0^{n-2}B_1(A_0+A_1b)}
\\
b_{0,1}^\prime&=&\frac{A_0b_{0,1}+2A_1b_{1,1}}{A_0^{n-2}(A_0+A_1b)},\\
b_{1,1}^\prime&=&\frac{B_1b_{1,1}}{A_0^{n-2}(A_0+A_1b)},\\
b_{1,2j}^{\prime }&=& \frac{1}{B_{1}(A_0+A_1b)}\left(\sum_{k=1}^{n-1}%
\sum_{l=2j}^{n-k-1}(-1)^{k-1}\,A_{0}^{1+2j-n}B_{k}B_{l-2j+1}b_{1,k+l-1}+\sum_{k=1}^{n-2}%
(-1)^k\,A_{0}^{1+2j-n}B_{k}B_{n-k-2j+1}b\right),
\end{eqnarray*}
where $l+k\neq n.$
\begin{proof}
%
%
%
%

Consider the product $[f(e_0), f(e_0)]=b_{0,0}^\prime f(e_n).$
Equating the coefficients of $e_n$ in it we get
$$A_0^2b_{0,0}+A_0A_1b_{0,1}+A_1^2b_{1,1}=b_{0,0}^\prime A_0^{n-2}B_1(A_0+A_1b).$$

$$ \mbox{Then}\ \ b_{0,0}^\prime=\frac{A_0^2b_{0,0}+A_0A_1b_{0,1}+A_1^2b_{1,1}}{A_0^{n-2}B_1(A_0+A_1b)}.$$

The product $[f(e_1), f(e_1)]=b_{1,1}^\prime f(e_n)$ yields
$$b_{1,1}^\prime=\frac{B_1b_{1,1}}{A_0^{n-2}(A_0+A_1b)}.$$

Consider the equality
$$b_{0,1}^\prime f(e_n)=[f(e_1), f(e_0)]+[f(e_0), f(e_1)].$$

Then $b_{0,1}^\prime
A_0^{n-2}B_1(A_0+A_1b)=A_0B_1b_{0,1}+2A_1B_1b_{1,1}$  and it
implies that
$$b_{0,1}^\prime=\frac{A_0b_{0,1}+2A_1b_{1,1}}{A_0^{n-2}(A_0+A_1b)}.$$

%

%
%
According to Proposition 3.3.
\begin{eqnarray*}e_{0}^{\prime }&=&A_0\,e_0 + A_1\,e_1 \\[1mm]
e_{1}^{\prime }&=&B_1\,e_1+B_2\,e_2+...+B_{n-2}\,e_{n-2}\\[1mm]
 e_{i}^{\prime
}&=&\sum_{k=i}^{n-1}A_{0}^{i-1}B_{k-i+1}e_{i}+(\ast
)e_{n}~\ \ \ \ \ \ \ \ \ \ \ \ \ \ \ 2\leq i\leq n-1\\[1mm]
e_{n}^{\prime }&=& A_{0}^{n-2}B_{1}(A_0+A_1b) \,e_n ,\end{eqnarray*}

then
\begin{eqnarray*}
[e'_i,e'_j]&=&[\sum_{k=i}^{n-1}A_{0}^{i-1}B_{k-i+1}e_{k}+(\ast
)e_{n},\sum_{l=j}^{n-1}A_{0}^{j-1}B_{l-j+1}e_{l}+(\ast )e_{n}],\\[1mm]
&=&[\sum_{k=i}^{n-1}A_{0}^{i-1}B_{k-i+1}e_{k},
\sum_{l=j}^{n-1}A_{0}^{j-1}B_{l-j+1}e_{l}]\\[1mm]
&=&\sum_{k=i}^{n-1}\sum_{l=j}^{n-1}A_{0}^{i+j-2}B_{k-i+1}B_{l-j+1}[e_{k},e_{l}]\\[1mm]
&=&\sum_{k=i}^{n-1}\sum_{l=j}^{n-k}A_{0}^{i+j-2}B_{k-i+1}B_{l-j+1}b_{k,l}\
\ e_{n}.
\end{eqnarray*}

Hence the equality
\begin{eqnarray*}
\ \ \ b_{i,j}^{\prime }\,e_{n}^{\prime }\
&=&[e_{i}^{\prime
},e_{j}^{\prime }]\end{eqnarray*}

gives the relation
\begin{eqnarray*}\ \ b_{i,j}^{\prime
}A_0^{n-2}B_{1}(A_0+A_1b)&=&\sum_{k=i}^{n-1}
\sum_{l=j}^{n-k}A_{0}^{i+j-2}B_{k-i+1}B_{l-j+1}b_{k,l},\end{eqnarray*}

and then
\begin{eqnarray*}
b_{i,j}^{\prime }&=&\frac{1}{B_{1}(A_0+A_1b)}(\sum_{k=i}^{n-1}%
\sum_{l=j}^{n-k}A_{0}^{i+j-n}B_{k-i+1}B_{l-j+1}b_{k,l}).
\end{eqnarray*}

from above \textbf{lemma 3.1.} if $b_{i,j}\neq0$ can be
representative as $b_{1,2j}$ so final formula will be :

$$b_{1,2j}^{\prime }= \frac{1}{B_{1}(A_0+A_1b)}\left(\sum_{k=1}^{n-1}%
\sum_{l=2j}^{n-k-1}(-1)^{k-1}\,A_{0}^{1+2j-n}B_{k}B_{l-2j+1}b_{1,k+l-1}+\sum_{k=1}^{n-2}%
(-1)^k\,A_{0}^{1+2j-n}B_{k}B_{n-k-2j+1}b\right)$$ where $l+k\neq
n.$

\end{proof}
\end{lm}

The next sections deal with the applications of the results of the
previous section to the classification problem of $CE(\mu_n)$ at
$n=$5 -- 9. It should be mentioned that the classifications of all
complex nilpotent Leibniz algebras in dimensions at most 4 have
been done before in \cite{AOR1}.

Here to classify algebras from $CE(\mu_n)$ in each fixed
dimensional case we represent it as a disjoin union of its
subsets. Some of these subsets are single orbits and the others
contain infinitely many orbits. In the last case we give invariant
functions to discern the orbits.

To simplify calculation we will introduced the following
notations: $$\Delta ={{b_{0,1}^{2}}-4\,b_{0,0}\,b_{1,1}} \ \ \mbox{and} \ \ \Delta
^{\prime} ={b_{0,1}^{{\prime }2}}-4\,b_{0,0}^{ \prime}\,b_{1,1}^{
\prime}. $$

\subsection{ Central extension for $4$-dimensional Lie algebra $CE(\mu_4)$ }

%
%
This section is devoted to the complete classification of
$CE(\mu_4).$ According to our notations the elements of
$CE(\mu_4)$ will be denoted by $L(\alpha),$ where
$\alpha=(b_{0,0},b_{0,1},b_{1,1},b_{1,2}).$ Note that in this case
$n$ is even then $b=0$ (see the multiplication table of
$CE(\mu_n).$)


\begin{theor}
\bigskip\ (Isomorphism criterion for $CE(\mu_4)$) Two filiform Leibniz algebras
$L(\alpha)$ and $L(\alpha')$ from $CE(\mu_4)$ are isomorphic iff
there exist $ A_0,A_1,B_1\in \mathbb{C}:$ such that $A_0B_1\neq 0$
and the following equalities hold:

\begin{eqnarray}
b_{0,0}^\prime&=&\frac{A_0^2b_{0,0}+A_0A_1b_{0,1}+A_1^2b_{1,1}}{A_0^{3}B_1},\\
b_{1,1}^\prime&=&\frac{B_1b_{1,1}}{A_0^{3}},\\
b_{0,1}^\prime&=&\frac{A_0b_{0,1}+2A_1b_{1,1}}{A_0^{3}},\\
b_{1,2}'&=&\frac{B_1b_{1,2}}{A_0^2}.
 \end{eqnarray}

\begin{proof}\emph{}
\textquotedblleft \ If \textquotedblright \ part due to Lemma 3.2.

\textquotedblleft \ Only if part.\textquotedblright \

Let the equations (3) -- (6) hold. Then the following basis
changing is adapted and it transforms $L(\alpha)$ to $L(\alpha')$

\begin{eqnarray*}
e'_0&=&A_{{0}}e_{{0}}+A_{{1}}e_{{1}},
\\
e'_1&=&B_{{1}}e_{{1}},
\\
e'_2&=&A_{{0}}B_{{1}}e_{{2}}+A_{{1}}B_{{1}}b_{{1,1}}e_{{4}},
\\
e'_3&=& {A_{{0}}}^{2}B_{{1}}e_{{3}}-A_{{1}}A_{{0}}B_{{1}}b_{{1,2}}e_{{4}},\\
e'_4&=&{A_{{0}}}^{3}B_{{1}}e_{{4}}.
\end{eqnarray*}
Indeed,
\begin{eqnarray*}
[e'_0,e'_0]&=&{A_{{0}}}^{2}b_{{0,0}}e_{{4}}+A_{{0}}A_{{1}} \left(
-e_{{2}}+ b_{{0,1}}e_{{4}} \right)
+A_{{1}}A_{{0}}e_{{2}}+{A_{{1}}}^{2} b_{{1,1}}e_{{4}}\\
&=& \frac{\left( {A_{{0}}}^{2}b_{{0,0}}+A_{{0}}A_{{1}}b_{{0,1}
}+{A_{{1}}}^{2}b_{{1,1}} \right)}{A_0^{3}\,B_1}  \ \ A_0^{3}\,B_1 e_{{4}}= b'_{0,0}\, e'_4.
\end{eqnarray*}

By the same steps we can get the second equation
\begin{eqnarray*}
[e'_0,e'_1]&=&
-A_{{0}}B_{{1}}e_{{2}}+A_{{0}}B_{{1}}b_{0,1}e_4+A_{{1}}B_{{1}}
b_{{1,1}} e_{{4}}
\\
&=& -A_{{0}}B_{{1}}e_{{2}}-A_{{1}}B_{{1}} b_{{1,1}}
e_{{4}}+A_{{0}}B_{{1}}b_{0,1}e_4+2\,A_{{1}}B_{{1}} b_{{1,1}}
e_{{4}}\\
&=&-e'_2+B_{{1}}\,\left( A_{{0}}b_{{0,1}}+2\,A_{{1}} b_{{1,1}}
\right) e_{{4}}\\
&=&-e'_2+b'_{0,1}\,A^3_0\,B_1\,e_4=-e'_2+b'_{0,1}\,e_4
\end{eqnarray*}

\begin{eqnarray*}
[e'_1,e'_1]&=&{B_{{1}}}^{2}b_{{1,1}}e_{{4}}\\
&=& A^3_0{B_{{1}}}\,b'_{{1,1}}e_{{4}}=b'_{{1,1}}\,e_{{4}}
\end{eqnarray*}

\begin{eqnarray*}
[e'_1,e'_2]&=&{B_{{1}}}^{2}A_{{0}}b_{{1,2}}e_{{4}}\\
&=&A^3_0{B_{{1}}}\,b'_{{1,2}}e_{{4}}=b'_{{1,2}}\,e_{{4}}
\end{eqnarray*}

\end{proof}
\end{theor}

\bigskip In this section we give a list of all algebras from $CE(\mu_4)$ .

Represent $CE(\mu_4)\ $as a union of the following subsets :

$U_{1}=\{L(\alpha)\in CE(\mu_4)\ :b_{1,1}\neq 0,b_{1,2} \neq 0\}$

$U_{2}=\{L(\alpha)\in CE(\mu_4)\ :b_{1,1}\neq 0,b_{1,2}=0,\Delta
\neq 0\}$

$U_{3}=\{L(\alpha)\in CE(\mu_4)\ :b_{1,1}\neq
0,b_{1,2}=\Delta=0\}$

$U_{4}=\{L(\alpha)\in CE(\mu_4)\
:b_{1,1}=0,b_{0,1}\neq0,b_{1,2}\neq0\}$

$U_{5}=\{L(\alpha)\in CE(\mu_4)\
:b_{1,1}=0,b_{0,1}\neq0,b_{1,2}=0\}$

$U_{6}=\{L\alpha)\in CE(\mu_4)\
:b_{1,1}=b_{0,1}=0,b_{0,0}\neq0,b_{1,2}\neq0\}$

$U_{7}=\{L(\alpha)\in CE(\mu_4)\
:b_{1,1}=b_{0,1}=0,b_{0,0}\neq0,b_{1,2}=0\}$

$U_{8}=\{L(\alpha)\in CE(\mu_4)\
:b_{1,1}=b_{0,1}=b_{0,0}=0,b_{1,2}\neq0\}$

$U_{9}=\{L(\alpha)\in CE(\mu_4)\
:b_{1,1}=b_{0,1}=b_{0,0}=b_{1,2}=0\}$

\begin{pr}\emph{}
\begin{enumerate}
\item \emph{Two algebras $L(\alpha)$ and $L(\alpha')$ from $U_{1}$
are isomorphic if and only if}
\begin{equation*} \left({\frac {b_{{1,2}}}{b_{{1,1}}} }\right)^{4}
  \Delta=\left({\frac {b'_{{1,2}}}{b'_{{1,1}}} }\right)^{4}
  \Delta'\end{equation*}

\item \emph{For any $\lambda$ from $ \mathbb{C}$ there exists
$L(\alpha)\in U_1:\ \ \ $}\begin{equation*} \left({\frac
{b_{{1,2}}}{b_{{1,1}}} }\right)^{4}
  \Delta=\lambda\end{equation*}.
\end{enumerate}
Then algebras from the set$\ U_{1\ }$ can be parameterized as
$L(\lambda,0 ,1 ,1)$, $\ \ \ \lambda \in C$.

\begin{proof}
\bigskip $\Rightarrow $

Let $L(\alpha)$ and $L(\alpha')$ be isomorphic. Then due to
theorem 3.1 there are a complex numbers
\bigskip\ $A_{0},A_{1}\ $and $B_{1\ }:$ $\ \ A_{0}\,B_{1\ }\neq 0$
such that the action of the adapted group $G_{ad}$ can be
expressed by the following system of equations

\begin{eqnarray}
b_{0,0}^\prime&=&\frac{A_0^2b_{0,0}+A_0A_1b_{0,1}+A_1^2b_{1,1}}{A_0^{3}B_1},\\
b_{1,1}^\prime&=&\frac{B_1b_{1,1}}{A_0^{3}},\\
b_{0,1}^\prime&=&\frac{A_0b_{0,1}+2A_1b_{1,1}}{A_0^{3}},\\
b_{1,2}'&=&\frac{B_1}{A_0^2}b_{1,2}.
\end{eqnarray}

Then the one can easy to say that:
 \begin{equation*} \left({\frac {b_{{1,2}}}{b_{{1,1}}} }\right)^{4}
  \Delta=\left({\frac {b'_{{1,2}}}{b'_{{1,1}}} }\right)^{4}
  \Delta'\end{equation*}

\bigskip $\Leftarrow $

Let suppose the equality \ \ \begin{equation*} \left({\frac
{b_{{1,2}}}{b_{{1,1}}} }\right)^{4}
  \Delta=\left({\frac {b'_{{1,2}}}{b'_{{1,1}}} }\right)^{4}
  \Delta'\end{equation*} \ \ holds

Consider the basis changing
\begin{eqnarray*}
e_{0}^{\prime }&=&\sum_{i=0}^{4}A_{i}e_{i}\\
e_{i}^{\prime }&=&\sum_{k=i}^{3}A_{0}^{i-1}B_{k-i+1}e_{i}+(\ast
)e_{4}~\ \ \ \ \ \ \ \ \ \ \ \ \ \ \ 1\leq i\leq 3
\end{eqnarray*}

 Where $A_{0}={\frac {b_{{1,1}}}{b_{{1,2}}}}$ , $ A_{1}={-\frac
{b_{{0,1}}}{2\,b_{{1,2}}}} $, and $B_{1}={\frac
{{b_{{1,1}}}^{2}}{{b_{{1,2}}}^{3}}} .$This changing leads
$L(\alpha)$to $L(\left({\frac {b_{{1,2}}}{b_{{1,1}}} }\right)^{4}
  \Delta,0,1,1)$

The basis changing
\begin{eqnarray*}
e_{0}^{\prime \prime}&=&\sum_{i=0}^{4}A_{i}'e_{i}'\\
e_{i}^{\prime\prime
}&=&\sum_{k=i}^{3}A_{0}'^{i-1}B_{k-i+1}'e_{i}'+(\ast )e_{4}'~\ \ \
\ \ \ \ \ \ \ \ \ \ \ \ 1\leq i\leq 3
\end{eqnarray*}
Where $A_{0}'={\frac {b_{{1,1}}'}{b_{{1,2}}'}}$ , $ A_{1}'={-\frac
{b_{{0,1}}'}{2\,b_{{1,2}}'}} $, and $B_{1}'={\frac
{{b_{{1,1}}'}^{2}}{{b_{{1,2}}'}^{3}}} .$ This changing leads
$L(\alpha')$ to $L(\left({\frac {b'_{{1,2}}}{b'_{{1,1}}}
}\right)^{4}
  \Delta',0,1,1)$

but by the hypothesis of the theorem
 \begin{equation*} \left({\frac {b_{{1,2}}}{b_{{1,1}}} }\right)^{4}
  \Delta=\left({\frac {b'_{{1,2}}}{b'_{{1,1}}} }\right)^{4}
  \Delta'\end{equation*}

 so $L(\alpha)$ and $L(\alpha')$ are
isomorphic to the same algebra and therefore they are isomorphic.
\end{proof}

\end{pr}


\begin{pr}\emph{}

\begin{enumerate}

\item\qquad Algebras from $U_2$ are isomorphic to L(1,0,1,0);

\item\qquad Algebras from $U_3$ are isomorphic to L(0,0,1,0);

\item\qquad Algebras from $U_4$ are isomorphic to L(0,1,0,1);

\item\qquad Algebras from $U_5$ are isomorphic to L(0,1,0,0);

\item\qquad Algebras from $U_6$ are isomorphic to L(1,0,0,1);

\item\qquad Algebras from $U_7$ are isomorphic to L(1,0,0,0);

\item\qquad Algebras from $U_8$ are isomorphic to L(0,0,0,1);

\item\qquad Algebras from $U_{9}$ are isomorphic to L(0,0,0,0).
\end{enumerate}

\begin{proof}\emph{}

We will show that $U_2,...,U_9$ are single orbit .To show for each
subsets we find the corresponding basis changing leading to
indicated in representative

For $U_2$

 \begin{eqnarray*}
e_2'&=&A_{{0}}B_{{1}} e_{{2}}+A_{{1}}B_{{1}}b_{{1,1}} e_{{4}},\\
e_3'&=& {A_{{0}}}^{2}B_{{1}}e_{{3}},\\
e_4'&=&{A_{{0}}}^{3}B_{{1}}e_{{4}}
\end{eqnarray*}

 where \begin{equation*}A_0=\frac{\Delta^\frac{1}{4}}{\sqrt{2}}\mathrm{,}A_1=\frac{-b_{0,1}\,\Delta^\frac{1}{4}}{2\sqrt{2}\,b_{1,1}} \mathrm{and}
 \ B_1=\frac{\Delta^\frac{3}{4}}{2\sqrt{2}\,b_{1,1}}\end{equation*}

 For $U_3$

 \begin{eqnarray*}
e_2'&=&A_{{0}}B_{{1}} e_{{2}}+ A_{{1}}B_{{1}}b_{{1,1}}
 e_{{4}},\\
e_3'&=& {A_{{0}}}^{2}B_{{1}}e_{{3}},\\
e_4'&=&{A_{{0}}}^{3}B_{{1}}e_{{4}}
\end{eqnarray*}

 where \begin{equation*}\ \ A_0\in \mathbb{C^*} \ \ \mathrm{,}\ \ A_1={\frac {-A_{{0}}b_{{0,1}}}{2b_{{1,1}}}}
 \mathrm{and}
 \ B_1={\frac {{A_{{0}}}^{3}}{b_{{1,1}}}}\end{equation*}

 For $U_4$

 \begin{eqnarray*}
e_2'&=&A_{{0}}B_{{1}} e_{{2}},\\
e_3'&=& {A_{{0}}}^{2}B_{{1}}e_{{3}}-A_{{0}}A_{{1}}B_{{1}}b_{{1,2}} e_{{4}},\\
e_4'&=&{A_{{0}}}^{3}B_{{1}}e_{{4}}
\end{eqnarray*}

 where \begin{equation*}A_0=\sqrt
{b_{{0,1}}}\mathrm{,}A_1=-{\frac {b_{{0,0}}}{\sqrt {b_{{0,1}}}}} \
\mathrm{and} \ B_1={\frac {b_{{0,1}}}{b_{{1,2}}}}\end{equation*}

 For $U_5$

 \begin{eqnarray*}
e_2'&=&A_{{0}}B_{{1}} e_{{2}}, \\
e_3'&=& {A_{{0}}}^{2}B_{{1}}e_{{3}},\\
e_4'&=&{A_{{0}}}^{3}B_{{1}}e_{{4}}
\end{eqnarray*}

 where \begin{equation*}A_0=\sqrt {b_{{0,1}}}\mathrm{,}A_1=-{\frac {b_{{0,0}}}{\sqrt {b_{{0,1}}}}} \ \mathrm{and}
 \ \ B_1\in \mathbb{C^*}\end{equation*}

For $U_6$

 \begin{eqnarray*}
e_2'&=&A_{{0}}B_{{1}} e_{{2}},\\
e_3'&=& {A_{{0}}}^{2}B_{{1}}e_{{3}}
-A_{{0}}A_{{1}}B_{{1}}b_{{1,2}}e_{{4}},\\
e_4'&=&{A_{{0}}}^{3}B_{{1}}e_{{4}}
\end{eqnarray*}

 where \begin{equation*}A_0=\sqrt [3]{b_{{0,0}}b_{{1,2}}}\mathrm{,}\ \ A_1\in \mathbb{C}\ \  \mathrm{and}
 \ B_1={\frac {b_{{0,0}}}{\sqrt [3]{b_{{0,0}}b_{{1,2}}}}}\end{equation*}

 For $U_7$

 \begin{eqnarray*}
e_2'&=&A_{{0}}B_{{1}} e_{{2}},\\
e_3'&=& {A_{{0}}}^{2}B_{{1}}e_{{3}},\\
e_4'&=&{A_{{0}}}^{3}B_{{1}}e_{{4}}
\end{eqnarray*}

 where \begin{equation*}\ \ A_0\in \mathbb{C^*} \ \ \mathrm{,}\ \ \ A_1\in \mathbb{C}\ \
 \mathrm{and}
 \ B_1={\frac {b_{{0,0}}}{A_{{0}}}}\end{equation*}

 For $U_8$

 \begin{eqnarray*}
e_2'&=&A_{{0}}B_{{1}} e_{{2}},\\
e_3'&=& {A_{{0}}}^{2}B_{{1}}e_{{3}}-A_{{0}}A_{{1}}B_{{1}}b_{{1,2}}e_{{4}},\\
e_4'&=&{A_{{0}}}^{3}B_{{1}}e_{{4}}
\end{eqnarray*}

 where \begin{equation*}\ \ A_0\in \mathbb{C^*} \ \ \mathrm{,}\ \ \ A_1\in \mathbb{C}\ \
 \mathrm{and}
 \ B_1={\frac {{A_{{0}}}^{2}}{b_{{1,2}}}}\end{equation*}

 For $U_{9}$

 \begin{eqnarray*}
e_2'&=&A_{{0}}B_{{1}} e_{{2}},\\
e_3'&=& {A_{{0}}}^{2}B_{{1}}e_{{3}},\\
e_4'&=&{A_{{0}}}^{3}B_{{1}}e_{{4}}
\end{eqnarray*}

 where \begin{equation*}\ \ A_0,\, B_1\in \mathbb{C^*} \ \ \mathrm{,}\ \ \ A_1\in \mathbb{C}\ \
 \mathrm{and}
 \end{equation*}

\end{proof}

\end{pr}

\subsection{ Central extension for $5$-dimensional Lie algebra $CE(\mu_5)$}

%

From Leibniz Identity we can show that $b=b_{2,3}=b_{1,4}$ \\
Further the elements of $CE(\mu_5)$ will be denoted by
$L(\alpha),$ where $\alpha=(b_{0,0},b_{0,1},b_{1,1},b_{1,2},b)$
meaning that they are depending on parameters $
b_{0,0},b_{0,1},b_{1,1},b_{1,2},b.$


\begin{theor}
\bigskip\ (Isomorphism criterion for $CE(\mu_5)$) Two  filiform Leibniz algebras
$L(\alpha)$ and $L(\alpha')$ from $CE(\mu_5)$ are isomorphic iff $
\exists \ A_0,A_1,B_1\in \mathbb{C}:$ such that $A_0B_1\left(
A_{{0}}+A_{{1}}b
 \right)\neq 0 $ ,and the following equalities hold:

\begin{eqnarray}
b_{0,0}^\prime&=&{\frac
{{A^{2}_{{0}}}b_{{0,0}}+A_{{0}}A_{{1}}b_{{0,1}}+{A^{2}_{{
1}}}b_{{1,1}}}{{A^{3}_{{0}}}B_{{1}} \left( A_{{0}}+A_{{1}}b
 \right) }}
,\\
b_{0,1}^\prime&=&{\frac
{A_{{0}}b_{{0,1}}+2\,A_{{1}}b_{{1,1}}}{{A^{3}_{{0}}}
 \left( A_{{0}}+A_{{1}}b\right) }}
,\\
b_{1,1}^\prime&=&{\frac {B_{{1}}b_{{1,1}}}{{A^{3}_{{0}}} \left(
A_{{0}}+A_{{1}}b
 \right) }}
,\\
b_{1,2}'&=&{\frac
{{B_{{1}}}^{2}b_{{1,2}}+(-2\,B_{{1}}B_{{3}}+{B^{2}_{{2}}})b}{{
A_{{0}}}^{2}B_{{1}} \left( A_{{0}}+A_{{1}}b \right) }}
,\\
b'&=&\frac {{B_{{1}}}b}{{ \left( A_{{0}}+A_{{1}}b \right) }}.
 \end{eqnarray}

\begin{proof}\ \ \ \ \ \ \ \ \ \ \ \ \ \ \ \ \ \ \ \ \ \ \ \ \ \ \
\

\end{proof}

\end{theor}

In this section we give a list of all algebras from $CE(\mu_5)$ .


Represent $CE(\mu_5)\ $as a union of the following subsets: \ \

$U_{1}=\{L(\alpha)\in CE(\mu_5)\ :b\neq 0,b_{1,1} \neq 0\}$

$U_{2}=\{L(\alpha)\in CE(\mu_5)\ :b \neq 0,b_{1,1}=0,b_{0,1} \neq
0\}$

$U_{3}=\{L(\alpha)\in CE(\mu_5)\ :b \neq 0,b_{1,1}=b_{0,1}=
0,b_{0,0}\neq0\}$

$U_{4}=\{L(\alpha)\in CE(\mu_5)\ :b \neq
0,b_{1,1}=b_{0,1}=b_{0,0}=0\}$

$U_{5}=\{L(\alpha)\in CE(\mu_5)\ :b=0, b_{1,1}\neq0,b_{1,2}\neq
0\}$

$U_{6}=\{L(\alpha)\in CE(\mu_5)\
:b=0,b_{1,1}\neq0,b_{1,2}=0,\Delta \neq 0\}$

$U_{7}=\{L(\alpha)\in CE(\mu_5)\ :b=0,
b_{1,1}\neq0,b_{1,2}=\Delta=0\}$

$U_{8}=\{L(\alpha)\in CE(\mu_5)\ :b=b_{1,1}=0,b_{0,1}\neq
0,b_{1,2}\neq 0\}$

$U_{9}=\{L(\alpha)\in CE(\mu_5)\ :b=b_{1,1}=0,b_{0,1}\neq
0,b_{1,2}=0\}$

$U_{10}=\{L(\alpha)\in CE(\mu_5)\
:b=b_{1,1}=b_{0,1}=0,b_{0,0}\neq0,b_{1,2}\neq 0\}$

$U_{11}=\{L(\alpha)\in CE(\mu_5)\
:b=b_{1,1}=b_{0,1}=0,b_{0,0}\neq0,b_{1,2}=0\}$

$U_{12}=\{L(\alpha)\in CE(\mu_5)\
:b=b_{1,1}=b_{0,1}=b_{0,0}=0,b_{1,1}\neq0\}$

$U_{13}=\{L(\alpha)\in CE(\mu_5)\
:b=b_{1,1}=b_{0,1}=b_{0,0}=b_{1,1}=0\}$

\begin{pr}\emph{}
\begin{enumerate}
\item \emph{Two algebras $L(\alpha)$ and $L(\alpha')$ from $U_1$
are isomorphic if and only if}
$$ {\frac {\Delta \,{b}^{2} }{ \left(b_{{0,1}}b- 2 \, b_{{1,1}}
\right) ^ {2}}}
 ={\frac {\Delta' \,{b}^{2} }{ \left(b_{{0,1}}b'- 2 \, b_{{1,1}}
\right) ^ {2}}}$$
 \item \emph{For any $\lambda$ from $
\mathbb{C}$ there exists $L(\alpha)\in U_1:\ \ \ $}
$${\frac {\Delta \,{b}^{2} }{ \left(b_{{0,1}}b- 2 \, b_{{1,1}}
\right) ^ {2}}}=\lambda$$ .
\end{enumerate}

Then algebras from the set$\ U_{1\ }$ can be parameterized as
$L(\lambda,0 ,1,0 ,1)$, $\ \ \ \lambda \in \mathbb{C}$.

\end{pr}

\begin{pr}\emph{}
\begin{enumerate}
\item \emph{Two algebras $L(\alpha)$ and $L(\alpha')$ from $U_{5}$
are isomorphic if and only if}
$$ \left({\frac{b_{1,2}}{b_{1,1}}}\right)^{6} \Delta=\left({\frac{b'_{1,2}}{b'_{1,1}}}\right)^{6} \Delta' $$

\item \emph{For any $\lambda$ from $ \mathbb{C}$ there exists
$L(\alpha)\in U_5:\ \ \ $}
$$\left({\frac{b_{1,2}}{4\,b_{1,1}}}\right)^{6} \Delta=\lambda$$
.\end{enumerate}

Then algebras from the set$\ U_{5\ }$ can be parameterized as
$L(\lambda,0 ,1,1,0)$, $\ \ \ \lambda \in \mathbb{C}$.

\end{pr}

\begin{pr}\emph{}
\begin{enumerate}

\item\qquad Algebras from $U_2$ are isomorphic to L(0,1,0,0,1);

\item\qquad Algebras from $U_3$ are isomorphic to L(1,0,0,0,1);

\item\qquad Algebras from $U_4$ are isomorphic to L(0,0,0,0,1);

\item\qquad Algebras from $U_6$ are isomorphic to L(1,0,1,0,0);

\item\qquad Algebras from $U_7$ are isomorphic to L(0,0,1,0,0);

\item\qquad Algebras from $U_8$ are isomorphic to L(0,1,0,1,0);

\item\qquad Algebras from $U_9$ are isomorphic to L(0,1,0,0,0);

\item\qquad Algebras from $U_{10}$ are isomorphic to L(1,0,0,1,0);

\item\qquad Algebras from $U_{11}$ are isomorphic to L(1,0,0,0,0);

\item\qquad Algebras from $U_{12}$ are isomorphic to L(0,0,0,1,0);

\item\qquad Algebras from $U_{13}$ are isomorphic to L(0,0,0,0,0).
\end{enumerate}

\end{pr}

\subsection{Central extension for $6$-dimensional Lie algebra $CE(\mu_6)$}

This section is devoted to the classification of $CE(\mu_6)$.

%
%

from Lemma (3.1) it is easy to prove $b_{{1,4}}=-b_{{2,3}}$

\begin{theor}
\bigskip\ (Isomorphism criterion for $CE(\mu_6)$) Two filiform Leibniz algebras
$\alpha=(b_{0,0},b_{0,1},b_{1,1},b_{1,2},b_{1,4})$ and
$\alpha'=(b_{0,0}',b_{0,1}',b_{1,1}',b_{1,2}',b_{1,4}')$ from
$CE(\mu_6)$ are isomorphic iff $\exists  A_0,A_1,B_1\in
\mathbb{C}:$ such that $A_0B_1\neq 0$ and the following equalities
hold:

\begin{eqnarray}
b_{0,0}^\prime&=&\frac{A^2_0b_{0,0}+A_0A_1b_{0,1}+A^2_1b_{1,1}}{A^{5}_0B_1},\\
b_{1,1}^\prime&=&\frac{B_1b_{1,1}}{A^{5}_0},\\
b_{0,1}^\prime&=&\frac{A_0b_{0,1}+2A_1b_{1,1}}{A^{5}_0},\\
b_{1,2}'&=&\frac{1}{A^4_0B_1}({B^{2}_{1}}b_{1,2}+(2B_1B_3-B^2_2)b_{1,4}),\\
b_{1,4}'&=&\frac{B_1}{A^2_0}b_{1,4}.
\end{eqnarray}

\begin{proof} see prove of theorem \textbf{3.1}
\end{proof}

\end{theor}


 In this section we give a list of all algebras from $CE(\mu_6)$ .

Let \ $ \Delta ={{b_{0,1}^{2}}-4\,b_{0,0}\,b_{1,1}}$ and $ \Delta
^{\prime} ={b_{0,1}^{{\prime }2}}-4\,b_{0,0}^{ \prime}\,b_{1,1}^{
\prime} $ Represent $CE(\mu_6)\ $as a union of the following
subsets\ \textbf{$\bigskip
CE(\mu_6)=\bigcup\limits_{i=1}^{13}U_{i}$} \ \ where

$U_{1}=\{L(\alpha)\in CE(\mu_6)\ :b_{1,1} \neq 0,b_{1,4} \neq 0\}$

$U_{2}=\{L(\alpha)\in CE(\mu_6)\ :b_{1,1} \neq 0,b_{1,4}=
0,b_{1,2} \neq0\}$

$U_{3}=\{L(\alpha)\in CE(\mu_6)\ :b_{1,1} \neq 0,b_{1,4} = b_{1,2}
=0, \Delta\neq 0\}$

$U_{4}=\{L(\alpha)\in CE(\mu_6)\ :b_{1,1} \neq 0,b_{1,4} = b_{1,2}
=\Delta=0\}$

$U_{5}=\{L(\alpha)\in CE(\mu_6)\ :b_{1,1}=0,b_{0,1}\neq 0, b_{1,4}
\neq 0\}$

$U_{6}=\{L(\alpha)\in CE(\mu_6)\ :b_{1,1}=0,b_{0,1}\neq 0, b_{1,4}
=0, b_{1,2}\neq0\}$

$U_{7}=\{L(\alpha)\in CE(\mu_6)\ :b_{1,1}=0,b_{0,1}\neq 0, b_{1,4}
=b_{1,2}=0\}$

$U_{8}=\{L(\alpha)\in CE(\mu_6)\ :b_{1,1}=b_{0,1}=0,b_{0,0}\neq 0
, b_{1,4} \neq0\}$

$U_{9}=\{L(\alpha)\in CE(\mu_6)\ :b_{1,1}=b_{0,1}=0,b_{0,0}\neq 0,
b_{1,4}=0 ,b_{1,2}\neq 0\}$

$U_{10}=\{L(\alpha)\in CE(\mu_6)\ :b_{1,1}=b_{0,1}=0,b_{0,0}\neq
0, b_{1,4}=b_{1,2}=0\}$

$U_{11}=\{L(\alpha)\in CE(\mu_6)\ :b_{1,1}=b_{0,1}=b_{0,0}=0,
b_{1,4}\neq0\}$

$U_{12}=\{L(\alpha)\in CE(\mu_6)\
:b_{1,1}=b_{0,1}=b_{0,0}=b_{1,4}=0,b_{1,2}\neq 0\}$

$U_{13}=\{L(\alpha)\in CE(\mu_6)\
:b_{1,1}=b_{0,1}=b_{0,0}=b_{1,4}=b_{1,2}=0\}$


\begin{pr}\emph{}
\begin{enumerate}
\item \emph{Two algebras $L(\alpha)$ and $L(\alpha')$ from $U_{1}$
are isomorphic if and only if}
$$ \left({\frac {b_{1,4}}{b_{1,1}}}\right)^{8} \Delta ^{3}
 =\left({\frac {b'_{1,4}}{b'_{1,1}}}\right)^{8} \Delta'^{3}$$

\item \emph{For any $\lambda$ from $ \mathbb{C}$ there exists
$L(\alpha)\in U_1:\ \ \ \ \left({\frac
{b_{1,4}}{b_{1,1}}}\right)^{8} \Delta ^{3}=\lambda$.}
\end{enumerate}
\bigskip Then algebras from the set$\ U_{1\ }$ can be parameterized
as $L(\lambda,0 ,1,0,1)$, $\ \ \ \lambda \in C$.

\end{pr}


\begin{pr}\emph{}
\begin{enumerate}
\item \emph{Two algebras $L(\alpha)$ and $L(\alpha')$ from $U_{2}$
are isomorphic if and only if}
$$ \left({\frac {b_{1,2}}{b_{1,1}}}\right)^{8} \Delta
 =\left({\frac {b'_{1,2}}{b'_{1,1}}}\right)^{8} \Delta' $$
\item \emph{For any $\lambda$ from $ \mathbb{C}$ there exists
$L(\alpha)\in U_2:\ \ \ \left({\frac
{b_{1,2}}{b_{1,1}}}\right)^{8} \Delta=\lambda$.}
\end{enumerate}
 Then algebras from the set$\ U_{2\ }$ can be
parameterized as $L(\lambda,0 ,1 ,1,0)$, $\ \ \ \lambda \in C$.

\end{pr}


\begin{pr}\emph{}
\begin{enumerate}

\item\qquad Algebras from  $U_3$ are isomorphic to L(1,0,1,0,0);

\item\qquad Algebras from  $U_4$ are isomorphic to L(0,0,1,0,0);

\item\qquad Algebras from  $U_5$ are isomorphic to L(0,1,0,0,1);

\item\qquad Algebras from  $U_6$ are isomorphic to L(0,1,0,1,0);

\item\qquad Algebras from  $U_7$ are isomorphic to L(0,1,0,0,0);

\item\qquad Algebras from  $U_8$ are isomorphic to L(1,0,0,0,1);

\item\qquad Algebras from  $U_9$ are isomorphic to L(1,0,0,1,0);

\item\qquad Algebras from  $U_{10}$ are isomorphic to
L(1,0,0,0,0);

\item\qquad Algebras from  $U_{11}$ are isomorphic to
L(0,0,0,0,1);

\item\qquad Algebras from  $U_{12}$ are isomorphic to
L(0,0,0,1,0);

\item\qquad Algebras from  $U_{13}$ are isomorphic to
L(0,0,0,0,0).

\end{enumerate}

\end{pr}

\subsection{ Central extension for $7$-dimensional Lie algebra $CE(\mu_7)$}

%

From Leibniz Identity we can show that $b_{2,5}=-b_{3,4}=b , b_{2,3}=-b_{1,4}$ \\
Further the elements of $CE(\mu_7)$ will be denoted by $L(\alpha)$
where $\alpha=(b_{0,0},b_{0,1},b_{1,1},b_{1,2},b_{1,4},b)$ meaning
that they are depending on parameters $
b_{0,0},b_{0,1},b_{1,1},b_{1,2},b_{1,4},b.$


\begin{theor}
\bigskip\ (Isomorphism criterion for $CE(\mu_7)$) Two filiform Leibniz algebras
$L(b_{0,0},b_{0,1},b_{1,1},b_{1,2},b_{1,4},b)$ and
$L(b_{0,0}',b_{0,1}',b_{1,1}',b_{1,2}',b_{1,4}',b')$ from
$CE(\mu_7)$ are isomorphic iff $ \exists \ A_0,A_1,B_1\in
\mathbb{C}:$ such that $A_0B_1\left( A_{{0}}+A_{{1}}b
 \right)\neq 0 $ ,and the following equalities hold:

\begin{eqnarray}
b_{0,0}^\prime&=&{\frac
{{A_{{0}}}^{2}b_{{0,0}}+A_{{0}}A_{{1}}b_{{0,1}}+{A_{{
1}}}^{2}b_{{1,1}}}{{A_{{0}}}^{5}B_{{1}} \left( A_{{0}}+A_{{1}}b
 \right) }}
,\\
b_{0,1}^\prime&=&{\frac
{A_{{0}}b_{{0,1}}+2\,A_{{1}}b_{{1,1}}}{{A_{{0}}}^{5}
 \left( A_{{0}}+A_{{1}}b\right) }}
,\\
b_{1,1}^\prime&=&{\frac {B_{{1}}b_{{1,1}}}{{A_{{0}}}^{5} \left(
A_{{0}}+A_{{1}}b
 \right) }}
,\\
b_{1,2}'&=&{\frac
{{B_{{1}}}^{2}b_{{1,2}}+(2\,B_{{1}}B_{{3}}-{B_{{2}}}^{2})b_{1,4}+(2B_{2}B_{4}-2B_{1}B_{5}-B_{3}^3)b}{2{
A_{{0}}}^{4}B_{{1}} \left( A_{{0}}+A_{{1}}b \right) }}
,\\
b_{1,4}'&=&\frac
{-{B_{{1}}}b_{{1,4}}+(-2\,B_{{1}}B_{{3}}+{B_{{2}}}^{2})b}{{
A_{{0}}}^{2}B_{{1}} \left( A_{{0}}+A_{{1}}b \right)},\\
b'&=&\frac{B\,b}{A_0+A_1b}.
 \end{eqnarray}

\begin{proof}\ \ \ \ \ \ \ \ \ \ \ \ \ \ \ \ \ \ \ \ \ \ \ \ \ \ \
\

Part \textquotedblright if \textquotedblleft. Let $L_1$ and $L_2$
from $CE(\mu_7)$ be isomorphic: $f:L_1 \cong L_2.$ We choose the
corresponding adapted basis $\{e_0, e_1,...,e_7,\}$ in $L_1$ and
$\{e_0',e_1'.., e'_7\}$ in $L_2.$ Then on this basis the algebras
well be presented as $L(\alpha)$  and $L(\alpha')$
\end{proof}

\end{theor}

\bigskip In this section we give a list of all algebras from $CE(\mu_7)$ .


Represent $CE(\mu_7)\ $as a union of the following subsets: \

$U_{1}=\{L(\alpha)\in CE(\mu_7)\ :b \neq 0,b_{1,1} \neq 0\}$

$U_{2}=\{L(\alpha)\in CE(\mu_7)\ :b \neq
0,b_{1,1}=0,b_{0,1}\neq0\}$

$U_{3}=\{L(\alpha)\in CE(\mu_7)\ :b \neq
0,b_{1,1}=b_{0,1}=0,b_{0,0}\neq0\}$

$U_{4}=\{L(\alpha)\in CE(\mu_7)\ :b \neq
0,b_{1,1}=b_{0,1}=b_{0,0}=0\}$

$U_{5}=\{L(\alpha)\in CE(\mu_7)\ :b=0,b_{1,4}\neq0,b_{1,1}\neq0
\}$

$U_{6}=\{L(\alpha)\in CE(\mu_7)\ :b=0,b_{1,4}\neq0,b_{1,1}=0
,b_{0,1}\neq0 \}$

$U_{7}=\{L(\alpha)\in CE(\mu_7)\
:b=0,b_{1,4}\neq0,b_{1,1}=b_{0,1}=0 , b_{0,0}\neq 0\}$

$U_{8}=\{L(\alpha)\in CE(\mu_7)\
:b=0,b_{1,4}\neq0,b_{1,1}=b_{0,1}=b_{0,0}=0 \}$

$U_{9}=\{L(\alpha)\in CE(\mu_7)\
:b=b_{1,4}=0,b_{1,2}\neq0,b_{1,1}\neq0\}$

$U_{10}=\{L(\alpha)\in CE(\mu_7)\
:b=b_{1,4}=0,b_{1,2}\neq0,b_{1,1}=0,b_{0,1}\neq 0\}$

$U_{11}=\{L(\alpha)\in CE(\mu_7)\
:b=b_{1,4}=0,b_{1,2}\neq0,b_{1,1}=b_{0,1}=0 , b_{0,0}\neq 0\}$

$U_{12}=\{L(\alpha)\in CE(\mu_7)\
:b=b_{1,4}=0,b_{1,2}\neq0,b_{1,1}=b_{0,1}=b_{0,0}=0\}$

$U_{13}=\{L(\alpha)\in CE(\mu_7)\
:b=b_{1,4}=b_{1,2}=0,b_{1,1}\neq0,\Delta\neq0\}$

$U_{14}=\{L(\alpha)\in CE(\mu_7)\
:b=b_{1,4}=b_{1,2}=0,b_{1,1}\neq0,\Delta=0\}$

$U_{15}=\{L(\alpha)\in CE(\mu_7)\
:b=b_{1,4}=b_{1,2}=b_{1,1}=0,b_{0,1}\neq0\}$

$U_{16}=\{L(\alpha)\in CE(\mu_7)\
:b=b_{1,4}=b_{1,2}=b_{1,1}=b_{0,1}=0,b_{0,0}\neq0\}$

$U_{17}=\{L(\alpha)\in CE(\mu_7)\
:b=b_{1,4}=b_{1,2}=b_{1,1}=b_{0,1}=b_{0,0}=0\}$
\begin{pr}\emph{}
\begin{enumerate}
\item \emph{Two algebras $L(\alpha)$ and $L(\alpha')$ from $U_{1}$
are isomorphic if and only if}
$$ \left({\frac {{b}}{
-2\,b_{{1,1}}+b_{{0,1} }b }}\right) ^{2}\Delta
 =\left({\frac {b'}{
-2\,b'_{{1,1}}+b'_{{0,1} }b' }}\right) ^{2}\Delta'$$
 \item
\emph{For any $\lambda$ from $ \mathbb{C}$ there exists
$L(\alpha)$ $ \in U_1:\ \ \ $}
$$ \left({\frac {{b}}{
-2\,b_{{1,1}}+b_{{0,1} }b}}\right) ^{2}\Delta=\lambda$$ .
\end{enumerate}
\bigskip Then algebras from the set$\ U_{1\ }$ can be parameterized
as $L(\lambda,0 ,1,0,0 ,1)$, $\ \ \ \lambda \in \mathbb{C}$.

\end{pr}


\begin{pr}\emph{}
\begin{enumerate}
\item \emph{Two algebras $L(\alpha)$ and $L(\alpha')$ from $U_{5}$
are isomorphic if and only if}
$$ \left({\frac {b_{{1,4}}}{b_{{1,1}}}}\right)^{10}
\Delta^3 =\left({\frac {b'_{{1,4}}}{b'_{{1,1}}}}\right)^{10}
\Delta'^3 $$ \item \emph{For any $\lambda$ from $ \mathbb{C}$
there exists $L(\alpha)$ $ \in U_5:\ \ \ $}
$$ \left({\frac {b_{{1,4}}}{b_{{1,1}}}}\right)^{10}
\Delta^3=\lambda$$ .
\end{enumerate}

\bigskip Then algebras from the set$\ U_{5\ }$ can be parameterized
as $L(\lambda,0 ,1,0,1,0)$, $\ \ \ \lambda \in \mathbb{C}$.

\end{pr}


\begin{pr}\emph{}
\begin{enumerate}
\item \emph{Two algebras $L(\alpha)$ and $L(\alpha')$ from $U_{9}$
are isomorphic if and only if}
$$ \left({\frac {b_{{1,2}}}{b_{{1,1}}}}\right)^{10}
\Delta^3 =\left({\frac {b'_{{1,2}}}{b'_{{1,1}}}}\right)^{10}
\Delta'^3 $$
 \item \emph{For any $\lambda$ from $ \mathbb{C}$
there exists $L(\alpha)$ $ \in U_9:\ \ \ $}
$$\left({\frac {b_{{1,2}}}{b_{{1,1}}}}\right)^{10}
\Delta^3=\lambda$$ .
\end{enumerate}

\bigskip Then algebras from the set$\ U_{9\ }$ can be parameterized
as $L(\lambda,0 ,1,1,0,0)$, $\ \ \ \lambda \in C$.

\end{pr}

\begin{pr} \emph{}
\begin{enumerate}

\item\qquad Algebras from $U_2$ are isomorphic to L(0,1,0,0,0,1);

\item\qquad Algebras from $U_3$ are isomorphic to L(1,0,0,0,0,1);

\item\qquad Algebras from $U_4$ are isomorphic to L(0,0,0,0,0,1);

\item\qquad Algebras from $U_6$ are isomorphic to L(0,1,0,0,1,0);

\item\qquad Algebras from $U_7$ are isomorphic to L(1,0,0,0,1,0);

\item\qquad Algebras from $U_8$ are isomorphic to L(0,0,0,0,1,0);

\item\qquad Algebras from $U_{10}$ are isomorphic to
L(0,1,0,1,0,0);

\item\qquad Algebras from $U_{11}$ are isomorphic to
L(1,0,0,1,0,0);

\item\qquad Algebras from $U_{12}$ are isomorphic to
L(0,0,0,1,0,0);

\item\qquad Algebras from $U_{13}$ are isomorphic to
L(1,0,1,0,0,0);

\item\qquad Algebras from $U_{14}$ are isomorphic to
L(0,0,1,0,0,0);

\item\qquad Algebras from $U_{15}$ are isomorphic to
L(0,1,0,0,0,0);

\item\qquad Algebras from $U_{16}$ are isomorphic to
L(1,0,0,0,0,0);

\item\qquad Algebras from $U_{17}$ are isomorphic to
L(0,0,0,0,0,0).

\end{enumerate}

\end{pr}

\subsection{ Central extension for $8$-dimensional Lie algebra $CE(\mu_8)$}


     It is easy to prove that $b_{{1,4}}=-b_{{3,2}}$ and
$b_{{1,6}}=b_{{3,4}}=b_{{5,2}}$.The elements of $CE(\mu_8)$ will
denoted by $(b_{0,0},b_{0,1},b_{1,1},b_{1,2},b_{1,4},,b_{1,6})$
meaning that they are defined by parameters
$b_{0,0},b_{0,1},b_{1,1},b_{1,2},b_{1,4},,b_{1,6}$
\begin{theor}
(Isomorphism criterion for $CE(\mu_8)$) Two filiform Leibniz
algebras
$\alpha=(b_{0,0},b_{0,1},b_{1,1},b_{1,2},b_{1,4},b_{1,6})$ and
$\alpha'=(b_{0,0}',b_{0,1}',b_{1,1}',b_{1,2}',b_{1,4}',b'_{1,6})$
from $CE(\mu_8)$ are isomorphic iff $\exists  A_0,A_1,B_i\in
\mathbb{C}, \ \ \ \ 1\leq i\leq5$ such that $A_0B_1\neq 0$ and the
following equalities hold:
\begin{eqnarray} b_{0,0}^\prime&=&{\frac
{{A_{{0}}}^{2}b_{{0,0}}+A_{{0}}A_{{1}}
b_{{0,1}}+{A_{{1}}}^{2}b_{{1,1}}}{{A_{{0}}}^{7}B_{{1}}}} ,\\
b_{0,1}^\prime&=&{\frac
{2\,A_{{1}}b_{{1,1}}+A_{{0}}b_{{0,1}}}{{A_{{0}}}^{7}} } ,\\
b_{1,1}^\prime&=&{\frac {B_{{1}}b_{{1,1}}}{{A_{{0}}}^{7}}},\\
b_{1,2}'&=&{\frac {{B_{{1}}}^{2}b_{{1,2}}+ \left(
2\,B_{{1}}B_{{3}}-{B_{{2}} }^{2} \right) b_{{1,4}}+ \left(
2\,B_{{1}}B_{{5}}-2\,B_{{2}}B_{{4 }}+{B_{{3}}}^{2} \right)
b_{{1,6}}}{{A_{{0}}}^{6}B_{{1}}}},\\
b_{1,4}'&=&{\frac {{B_{{1}}}^{2}b_{{1,4}}+ \left(
2\,B_{{1}}B_{{3}}-{B_{{2}}}^{2}
 \right) b_{{1,6}}}{{A_{{0}}}^{4}B_{{
1}}}},\\
b_{{1,6}}'&=&{\frac {B_{{1}}b_{{1,6}}}{{A_{{0}}}^{2}}}.
\end{eqnarray}

\end{theor}

In this section we give a list of all algebras from $CE(\mu_8)$ .
Represent $CE(\mu_8)\ $as a union of the following subsets :\\

$U_{1}=\{L(\alpha)\in CE(\mu_8)\ :b_{1,6} \neq 0,b_{1,1} \neq 0\}$

$U_{2}=\{L(\alpha)\in CE(\mu_8)\ :b_{1,6} \neq
0,b_{1,1}=0,b_{0,1}\neq0\}$

$U_{3}=\{L(\alpha)\in CE(\mu_8)\ :b_{1,6} \neq
0,b_{1,1}=b_{0,1}=0,b_{0,0}\neq0\}$

$U_{4}=\{L(\alpha)\in CE(\mu_8)\ :b_{1,6} \neq
0,b_{1,1}=b_{0,1}=b_{0,0}=0\}$

$U_{5}=\{L(\alpha)\in CE(\mu_8)\
:b_{1,6}=0,b_{1,4}\neq0,b_{1,1}\neq0 \}$

$U_{6}=\{L(\alpha)\in CE(\mu_8)\ :b_{1,6}=0,b_{1,4}\neq0,b_{1,1}=0
,b_{0,1}\neq0 \}$

$U_{7}=\{L(\alpha)\in CE(\mu_8)\
:b_{1,6}=0,b_{1,4}\neq0,b_{1,1}=b_{0,1}=0 , b_{0,0}\neq 0\}$

$U_{8}=\{L(\alpha)\in CE(\mu_8)\
:b_{1,6}=0,b_{1,4}\neq0,b_{1,1}=b_{0,1}=b_{0,0}=0 \}$

$U_{9}=\{L(\alpha)\in CE(\mu_8)\
:b_{1,6}=b_{1,4}=0,b_{1,2}\neq0,b_{1,1}\neq0\}$

$U_{10}=\{L(\alpha)\in CE(\mu_8)\
:b_{1,6}=b_{1,4}=0,b_{1,2}\neq0,b_{1,1}=0,b_{0,1}\neq 0\}$

$U_{11}=\{L(\alpha)\in CE(\mu_8)\
:b_{1,6}=b_{1,4}=0,b_{1,2}\neq0,b_{1,1}=b_{0,1}=0 , b_{0,0}\neq
0\}$

$U_{12}=\{L(\alpha)\in CE(\mu_8)\
:b_{1,6}=b_{1,4}=0,b_{1,2}\neq0,b_{1,1}=b_{0,1}=b_{0,0}=0\}$

$U_{13}=\{L(\alpha)\in CE(\mu_8)\
:b_{1,6}=b_{1,4}=b_{1,2}=0,b_{1,1}\neq0,\Delta\neq0\}$

$U_{14}=\{L(\alpha)\in CE(\mu_8)\
:b_{1,6}=b_{1,4}=b_{1,2}=0,b_{1,1}\neq0,\Delta=0\}$

$U_{15}=\{L(\alpha)\in CE(\mu_8)\
:b_{1,6}=b_{1,4}=b_{1,2}=b_{1,1}=0,b_{0,1}\neq0\}$

$U_{16}=\{L(\alpha)\in CE(\mu_8)\
:b_{1,6}=b_{1,4}=b_{1,2}=b_{1,1}=b_{0,1}=0,b_{0,0}\neq0\}$

$U_{17}=\{L(\alpha)\in CE(\mu_8)\
:b_{1,6}=b_{1,4}=b_{1,2}=b_{1,1}=b_{0,1}=b_{0,0}=0\}$
In $U_1$ the following proposition is holds

\begin{pr}\emph{}
\begin{enumerate}
\item \emph{Two algebras $L(\alpha)$ and $L(\alpha')$ from $U_1$
are isomorphic if and only if}
$$ \left({\frac {b_{1,6}}{b_{1,1}}}\right)^{12} \Delta^{5}
 =\left({\frac {b'_{1,6}}{b'_{1,1}}}\right)^{12} \Delta'^{5}$$
\item \emph{For any $\lambda$ from $ \mathbb{C}$ there exists
$L(b_{0,0},b_{0,1},b_{1,1},b_{1,2},b_{1,4},b_{1,6})\in U_1:$\\ $$
\ \ \ \left({\frac {b_{1,6}}{b_{1,1}}}\right)^{12}
\Delta^{5}=\lambda$$.} \end{enumerate} Then algebras from the
set$\ U_{1\ }$ can be parameterized as $L(\lambda,0 ,1 ,0,0,1)$,
$\ \ \ \lambda \in \mathbb{C}$.

\end{pr}

In $U_2$ the following proposition is holds
\begin{pr}\emph{}
\begin{enumerate}
\item \emph{Two algebras $L(\alpha)$ and $L(\alpha')$ from $U_5$
are isomorphic if and only if}
$$\left(\frac {b_{{1,1}}}{b_{{1,4}}}\right)^{4}\,\frac{1}{\Delta }
 =\left(\frac {b'_{{1,1}}}{b'_{{1,4}}}\right)^{4}\,\frac{ 1}{\Delta' }$$

\item \emph{For any $\lambda$ from $ \mathbb{C}$ there exists
$L(b_{0,0},b_{0,1},b_{1,1},b_{1,2},b_{1,4},b_{1,6})\in U_5:$\\ $$
\ \ \ \left(\frac
{b_{{1,1}}}{b_{{1,4}}}\right)^{4}\,\frac{1}{\Delta }=\lambda$$.}
\end{enumerate}
\bigskip Then algebras from the set$\ U_{5\ }$ can be parameterized
as $L(\lambda,0 ,1 ,0,1,0)$, $\ \ \ \lambda \in \mathbb{C}$.

\end{pr}

In $U_9$ the following proposition is holds
\begin{pr}\emph{}
\begin{enumerate}
\item \emph{Two algebras
$L(b_{0,0},b_{0,1},b_{1,1},b_{1,2},b_{1,4},b_{1,6})$ and
$L(b_{0,0}',b_{0,1}',b_{1,1}',b_{1,2}',b_{1,4}',b_{1,6}')$ from
$U_9$ are isomorphic if and only if}
 $$\left(\frac {b_{{1,2}}}{b_{{1,1}}}\right)^{12}\,{\Delta }
 =\left(\frac {b'_{{1,2}}}{b'_{{1,1}}}\right)^{12}\,{\Delta' }$$
 \item \emph{For any $\lambda$ from $ \mathbb{C}$ there exists $L(b_{0,0},b_{0,1},b_{1,1},b_{1,2},b_{1,4},b_{1,6})\in
U_3:$\\} $$ \ \ \ \left(\frac
{b_{{1,2}}}{b_{{1,1}}}\right)^{12}\,{\Delta }=\lambda$$.

\end{enumerate}

\bigskip Then algebras from the set$\ U_{9\ }$ can be parameterized as $L(\lambda,0 ,1 ,1,0,0)$, $\ \ \ \lambda\in \mathbb{C}$.

\end{pr}

\begin{pr} \emph{}
\begin{enumerate}

\item\qquad Algebras from $U_2$ are isomorphic to L(0,1,0,0,0,1);

\item\qquad Algebras from $U_3$ are isomorphic to L(1,0,0,0,0,1);

\item\qquad Algebras from $U_4$ are isomorphic to L(0,0,0,0,0,1);

\item\qquad Algebras from $U_6$ are isomorphic to L(0,1,0,0,1,0);

\item\qquad Algebras from $U_7$ are isomorphic to L(1,0,0,0,1,0);

\item\qquad Algebras from $U_8$ are isomorphic to L(0,0,0,0,1,0);

\item\qquad Algebras from $U_{10}$ are isomorphic to
L(0,1,0,1,0,0);

\item\qquad Algebras from $U_{11}$ are isomorphic to
L(1,0,0,1,0,0);

\item\qquad Algebras from $U_{12}$ are isomorphic to
L(0,0,0,1,0,0);

\item\qquad Algebras from $U_{13}$ are isomorphic to
L(1,0,1,0,0,0);

\item\qquad Algebras from $U_{14}$ are isomorphic to
L(0,0,1,0,0,0);

\item\qquad Algebras from $U_{15}$ are isomorphic to
L(0,1,0,0,0,0);

\item\qquad Algebras from $U_{16}$ are isomorphic to
L(1,0,0,0,0,0);

\item\qquad Algebras from $U_{17}$ are isomorphic to
L(0,0,0,0,0,0).

\end{enumerate}

\end{pr}

\end{document}